\documentclass[12pt, a4paper]{article}

\usepackage{amsmath}
\usepackage{amsthm}
\usepackage{amssymb}
\usepackage{mathrsfs}						
\usepackage{stmaryrd}

\theoremstyle{plain}

\newtheorem{theorem}{Theorem}[section]
\newtheorem{lemma}[theorem]{Lemma}
\newtheorem{corollary}[theorem]{Corollary}
\newtheorem{proposition}[theorem]{Proposition}

\theoremstyle{definition}
\newtheorem{definition}[theorem]{Definition}

\newtheorem{example}[theorem]{Example}

\newtheorem{remark}[theorem]{Remark}

\theoremstyle{remark}
\newtheorem*{claim}{Claim}

\numberwithin{figure}{section}           
\numberwithin{equation}{section}

\DeclareMathOperator{\aut}{Aut}

\DeclareMathOperator{\Th}{Th}

\DeclareMathOperator{\autfL}{Autf_L}
\DeclareMathOperator{\galL}{Gal_L}

\DeclareMathOperator{\stab}{Stab}

\DeclareMathOperator{\rad}{rad}

\DeclareMathOperator{\cf}{cf}

\DeclareMathOperator{\tp}{tp}

\DeclareMathOperator{\cl}{cl}
\DeclareMathOperator{\id}{id}

\newcommand{\C}{{\mathfrak C}}

\newcommand{\N}{{\mathbb N}}
\newcommand{\Z}{{\mathbb Z}}
\newcommand{\Q}{{\mathbb Q}}
\newcommand{\R}{{\mathbb R}}

\newcommand{\MM}{{\mathbb M}}

\newcommand{\F}{{\mathbb F}}

\newcommand{\A}{{\mathcal A}}

\newcommand{\G}{{\mathcal G}}

\newcommand{\U}{{\mathcal U}}

\newcommand{\La}{{\mathrm L}}
\newcommand{\KP}{{\mathrm {KP}}}
\newcommand{\Sh}{{\mathrm {Sh}}}

\begin{document}

\title{Model theoretic connected components of groups}
\author{Jakub Gismatullin\thanks{The author is supported by the Polish Goverment MNiSW grant N N201 384134.}}
\date{\today}

\maketitle

\begin{abstract}
We give a general exposition of model theoretic connected components of groups. We show that if a group $G$ has NIP, then there exists the smallest invariant (over some small set) subgroup of $G$ with bounded index (Theorem \ref{thm:sh}). This result extends theorem of Shelah from \cite{sh2}. We consider also in this context the multiplicative and the additive groups of some rings (including infinite fields).
\end{abstract}

\footnotetext{2010 Mathematics Subject Classification: primary 03C60, 20A15, secondary 03C45, 12L12}
\footnotetext{Key words and phrases: groups with NIP, model theoretic connected components, thick sets}

\section*{Introduction}

Let $(G,\cdot,\ldots)$ be a group with some additional first order structure. In model theory we consider several kinds of \emph{model-theoretic connected components} of $G$. Assuming that $G$ is saturated, let $A\subset G$ be small set of parameters, and define (see Definition \ref{def:kp}): $G^{0}_A$ --- the connected component of $G$ over $A$, $G^{00}_A$ --- the type-connected component of $G$ over $A$ and $G^{\infty}_A$ --- the $\infty$-connected component of $G$ over $A$. If for every small $A$, $G^{\infty}_A = G^{\infty}_{\emptyset}$, then we call $G^{\infty}_{\emptyset}$ the $\infty$-connected component of $G$ and denote it by $G^{\infty}$ (we also say that $G^{\infty}$ exists in this case). Similarly we define the type-connected component $G^{00}$ of $G$ and the connected component $G^0$ of $G$.

The principal goal of this paper is to understand these three kinds of connected components in groups in basic algebraic examples (additive and multiplicative groups of fields and rings) and in important class of theories (NIP). This work continues research started in \cite{gis}.

$G^{00}$ has been studied widely in model theory, especially for $G$ definable in $o$-minimal structures \cite{NIP, pillay2}, in $p$-adically closed fields \cite{op} and more generally in NIP structures \cite{sh1, NIP, NIP2}. $G^{\infty}$ was considered as $G^{000}$ in the first version of \cite{NIP} and in \cite{NIP2}. 

The significance of $G^{00}$ has emerged through an influential conjecture of A. Pillay \cite{pillay} which asserts that if $G$ is a definably compact group definable in a saturated $o$-minimal structure, then the quotient $G/G^{00}$ is, when equipped with the ``logic topology'', a compact real Lie group whose dimension (as a Lie group) equals the dimension of $G$ as a definable set in an $o$-minimal structure. This conjecture, for $o$-minimal expansions of real closed fields, was proved in full in \cite{NIP}.

One of the motivations for considering $G^0_{\emptyset}$, $G^{00}_{\emptyset}$ and $G^{\infty}_{\emptyset}$, is the interplay between them and strong types. In model theory, types of elements of $G$ can often be thought as orbits on $G$ of the automorphism group $\aut(G)$. Various strong types (Lascar strong types, Kim-Pillay strong types and Shelah strong types) correspond to orbits of some canonical subgroups of $\aut(G)$. For instance, Lascar strong types are (or correspond to) orbits of $\autfL(G) \lhd \aut(G)$, i.e. the subgroup of Lascar strong automorphisms, which is generated by $\bigcup\{\aut(G/M) : M\text{ an elementary submodel of }G\}$. Strong types are essential to stability and simplicity theory --see e.g. \cite{pillay_book}. The groups $G^0_{\emptyset}$, $G^{00}_{\emptyset}$ and $G^{\infty}_{\emptyset}$ are important for characterising strong types in many structures (e.g. NIP structures). Indeed, in \cite[Section 3]{gis} we investigated the following construction: consider the following 2-sorted structure $\G = (G,X,\cdot)$, where $\cdot\colon G \times X \to X$ is a regular action of $G$ on $X$, and $X$ is a predicate (on $G$ we take its original structure). Then, Lascar, Kim-Pillay and Shelah strong types on the sort $X$ correspond exactly to orbits of $G^0_{\emptyset}$, $G^{00}_{\emptyset}$ and $G^{\infty}_{\emptyset}$ respectively on $X$. Also, $G/G^{\infty}_{\emptyset}$ with ``the logic topology'' is a quasi-compact topological group \cite[Proposition 3.7]{gis}, which can be seen as a canonical subgroup of \emph{the Lascar group} $\galL(\G)$ of the structure $\G$ (see \cite[Section 3]{gis}). The Lascar group is an abstractly defined invariant of first order theories of classical mathematical content (see \cite{pillay}).

The paper is organised as follows.

In the first Section we introduce some basic notation and prerequisites. 

In the second section we give a precise definition of all considered connected components, we collect their basic properties and prove some topological result regarding the closure of the invariant subgroups of the Lascar group. 

In section three we introduce the notion of thick subset of a group. Using this notion and result from section two we give a new characterisation/description of $G^{00}_A$ and $G^{\infty}_A$.

Section four is devoted (using results from section three) to study connected components of additive and multiplicative groups of some class of rings (including finitely-dimensional algebras over fields, e.g. matrix rings $\MM_n(K)$).

In the last section we prove (Theorem \ref{thm:sh}) that if a group $G$ has NIP, then $G^{\infty}$ exists. This extends a result of Shelah from \cite{sh2}, which provide existence of $G^{\infty}$ under NIP and commutativity assumption.

We assume that the reader is familiar with basic notions of model theory. The model-theoretic background can be found in \cite{Ma, pillay_book}.

We would like to thank the referee for many helpful suggestions.

This paper is a part of author's Ph.D. thesis written under supervision on L. Newelski.

\section{Basic Notation and Prerequisites}

$(G,\cdot,\ldots)$ is always a group with an additional first order structure in language $L$. Sometimes we assume that $G$ is already a sufficiently saturated model ($\overline{\kappa}$-saturated and $\overline{\kappa}$-strongly homogeneous for some large cardinal $\overline{\kappa}$). Usually $A\subset G$ is a small set of parameters.

A symmetric formula $\varphi(x,y)\in L(A)$ is \emph{thick} \cite[Section 3]{ziegler} \cite[Definition 1.10]{pillay} if for some $n\in\N$, for every sequence $(a_i)_{i<n}$ from $G$ (we do not require $a_0,\dots, a_{n-1}$ to be pairwise distinct) there exist $i<j<n$ such that $\varphi(a_i,a_j)$. By $\Theta_A$ we denote the conjunction of all thick formulas over $A$:
\[\Theta_A(x,y) = \bigwedge_{\varphi\text{ thick}}\varphi(x,y).\]

By $E_{\La/A}$, $E_{\KP/A}$ and $E_{\Sh/A}$ we denote equivalence relations of equality of Lascar, Kim-Pillay and Shelah strong types over $A$ respectively. $E_{\La/A}$ is the finest bounded $A$-invariant equivalence relation, $E_{\KP/A}$ is the finest bounded $\bigwedge$-definable over $A$ equivalence relation and $E_{\Sh/A}$ is intersection of all $A$-definable finite equivalence relations. We will need the following facts \cite[Lemma 7]{ziegler}: $E_{\La/A}$ is the transitive closure of $\Theta_A$; if $\Theta_A(a,b)$, then there is a small model $M\supseteq A$ such that $a\underset{M}{\equiv}b$; if for some small model $M\supseteq A$ we have $a\underset{M}{\equiv}b$, then $\Theta_A^2(a,b)$, i.e. there is $c$ such that $\Theta_A(a,c)\land\Theta_A(c,b)$.

By $\autfL(G)$ be denote the group of Lascar strong automorphisms of $G$, i.e. group generated by pointwise stabilisers of elementary submodels \[\autfL(G) = \big\langle \aut(G/M) : M \prec G \big\rangle.\] The \emph{Lascar (Galois) group} of the theory $T = \Th(G,\cdot,\ldots)$ of $G$ is the quotient $\galL(T) = \aut(G)/ \autfL(G)$. $\galL(T)$ is a quasi-compact topological group with some natural topology (see e.g. \cite[Section 1]{gis} or \cite[Section 4]{ziegler}).

For $H < \aut(G)$, by $E_H$ we consider the orbit equivalence relation defined as follows: $E_H(a,b)$ if and only if there is some $f\in H$ with $a=f(b)$.

For a group $G$ and a binary relation $E$ on $G$ we define the set of $E$-commutators \[X_E = \{a^{-1}b : a,b\in G, E(a,b)\}\] and the $E$-commutant $G_E$ as the subgroup of $G$ generated by $X_E$ \[G_E=\langle X_E \rangle <G.\]

For $a,b\in G$ and $A,B\subseteq G$ we make the following conventions: $a^b = b^{-1}ab$ and $A^B = \bigcup_{a\in A,b\in B}a^b$.

The first point of the next remark will be crucial in the proof Theorem \ref{thm:sh}.

\begin{remark} \label{rem:e} Let $(G,\cdot)$ be a group.
\begin{enumerate}
\item[(1)] If $E=E_H$ for some $H<\aut(G,\cdot)$, then \[(X_{E_H})^G \subseteq X_{E_H}^2\] and therefore $G_{E_H}\lhd G$.
\item[(2)] If $E$ is $\emptyset$-invariant, then $X_E$ and $G_E$ are also $\emptyset$-invariant.
\item[(3)] If $E$ is bounded, then $G_E$ has bounded index in $G$, moreover $[G : G_E]\leq |G/E|$.
\end{enumerate}
\end{remark}
\begin{proof}
$(1)$ Let $a,x \in G$ and $h\in H$. Then \[(X_{E_H})^x \ni(a^{-1}h(a))^x =(ax)^{-1}h(a)x = ((ax)^{-1}h(ax)) (h(x)^{-1}x)\in X_{E_H}^2.\] $(2)$ is obvious. $(3)$ follows from the observation: if $a^{-1}b\notin G_E$, then $\neg E(a,b)$.  
\end{proof} 

Let $G^*$ be a sufficiently saturated extension of $(G,\cdot,\ldots)$. Denote $X_{E_{\La/A}}$ by $X_{\La/A}$ and $G^*_{E_{\La/A}}$ by $G^*_{\La/A}$.

\begin{remark} \label{def:la}
 $G^*_{\La/A}$ is generated by $X_{\Theta_A} = \{a^{-1}b : a,b\in G^*\text{ and } \Theta_A(a,b)\}$.
\end{remark}
\begin{proof} $E_{\La/A}$ is the transitive closure of $\Theta_A$.
\end{proof}

\section{Connected components}

In this section $(G,\cdot,\ldots)$ is always a group with an additional first order structure in language $L$. We assume that $G$ is already a monster model ($\overline{\kappa}$-saturated and $\overline{\kappa}$-strongly homogeneous for some large cardinal $\overline{\kappa}$).

Here we define connected components $G^{\infty}_A$, $G^{00}_A$ and $G^{0}_A$ of $G$ relative to the set of parameters $A$ and also absolute components $G^{\infty}$, $G^{00}$ and $G^{0}$. Our approach follows Casanovas \cite[Section 5, Section 7]{cas} seminar paper.

\begin{definition} \label{def:kp}
\begin{itemize}
\item[(1)] $G^{0}_{A}   =  \bigcap\{ H < G : H \text{ is } A \text{-definable and } [G:H]<\omega \}$,
\item[(2)] $G^{00}_{A}  = \bigcap\{ H < G : H \text{ is } \bigwedge \text{-definable over } A \text{ and } [G:H]<\overline{\kappa} \}$,
\item[(3)] $G^{\infty}_{A} =  \bigcap\{ H < G : H \text{ is } A \text{-invariant and } [G:H]<\overline{\kappa}\}$.
\item[(4)] If for every small set of parameters $A\subset G$ \[G^{\infty}_A = G^{\infty}_{\emptyset},\] then we say that $G^{\infty}$ exists and define it as $G^{\infty}_{\emptyset}$. Similarly we define existence of $G^{00}$ and $G^{0}$.
\end{itemize}
\end{definition}

In \cite{sh1} there is another definition of ``existence of $G^{00}$'' in more general context of type definable group $G$: $G^{00}$ exists if and only if an arbitrary intersection of type definable subgroups of $G$ with bounded index (i.e. index $<\overline{\kappa}$) is a type definable subgroup also with bounded index (note that type definability is always assumed to be over small set of parameters, since every subset $X\subseteq G$ is type definable over $G\setminus X$). Since $A$-type-definable and $\aut(G)$-invariant subsets of $G$ are type definable over $\emptyset$, both of these definitions are equivalent.

In the next lemma we collect the basic properties of the groups $G^{\infty}_A$, $G^{00}_A$ and $G^{0}_A$ ($(1)$ and $(3)$ are known, see e.g. \cite[Section 5]{cas}). Groups $G^{0}_A$ and $G^{00}_A$ are $\bigwedge$-definable over $A$, and $G^{\infty}_A$ is $A$-invariant. $G^{00}_A$ is the smallest $\bigwedge$-definable over $A$ subgroup of $G$ of bounded index and similarly $G^{\infty}_A$ is the smallest $A$-invariant subgroup of $G$ of bounded index.

\begin{lemma} \label{lem:def}
\begin{enumerate}
\item[(1)] $[G:G^{\infty}_{A}] \leq 2^{|L(A)|}$
\item[(2)] $G^{\infty}_{A} = G_{\La/A} = \langle X_{\Theta_A} \rangle$
\item[(3)] $G^{\infty}_{A} \subseteq G^{00}_{A} \subseteq G^{0}_{A}$ are normal subgroups of $G$.
\end{enumerate}
\end{lemma}
\begin{proof}
$(1)$, $(2)$ Since every $A$-invariant subset of $G$ is the union of some $\bigwedge$-definable over $A$ sets and since there are only boundedly many (i.e. $2^{|S_1(A)|} \leq 2^{2^{L(A)}}$) such sets, there are only boundedly many $A$-invariant subgroups of $G$. Thus $G^{\infty}_{A}$ has bounded index in $G$.

Consider the following equivalence relation: $E(a,b)\Leftrightarrow a^{-1}b\in G^{\infty}_{A}$. $E$ is an $A$-invariant bounded equivalence relation, so \[E_{\La/A} \subseteq E.\] Therefore $[G:G^{\infty}_{A}] = |G/E| \leq |G/E_{\La/A}| \leq 2^{|L(A)|}$ and by the definition of $E$, $G_{\La/A} \subseteq G^{\infty}_{A}$. On the other hand $G_{\La/A}$ is an $A$-invariant subgroup of $G$ of bounded index (because if $a^{-1}b\not\in G_{\La/A}$, then $\neg E_{\La/A}(a,b)$), so $G^{\infty}_{A} \subseteq G_{\La/A}$. Thus $G^{\infty}_{A} = G_{\La/A}$.

$(3)$ Note first the following useful remark from group theory: if $H$ is a subgroup of $G$, then \[\bigcap_{g \in G} H^g = \bigcap_{i<[G:H]} H^{g_i}, \] where $\{g_i\}_{i<[G:H]}$ is the set of representatives of all right cosets of $H$. We may therefore assume (by this remark), that $H \lhd G$ in the definitions of $G^{\infty}_{A}$, $G^{00}_{A}$ and $G^{0}_{A}$ (because e.g. if $H<G$ is $A$-definable of finite index, then $\bigcap_{g \in G} H^g \lhd G$ is also $A$-definable of finite index).
\end{proof}

If $H$ is definable or $\bigwedge$-definable subgroup $G$, then its trace on the set of types, well reflects some properties of $H$. For instance, if $H$ has finite index in $G$, then the same is true in every monster model of $\Th(G,\cdot,\ldots)$ ($\bigwedge$-definable subgroup of finite index is definable). In the next remark we prove the similar result for invariant subgroups of $G$.

\begin{remark} \label{rem:cor}
Let $H < G$ be an $A$-invariant subgroup of $G$ and $\widetilde{H} = \{\tp(h/A) : h \in H\}$. If $G'$ is elementary equivalent to $G$ over $A$, then \[\widetilde{H}(G') = \{g\in G' : \tp(g/A) \in \widetilde{H}\}\] is also a group, called \emph{the group corresponding to $H$ in $G'$}.

If the groups $G$ and $G'$ are two monster models of the same theory and $H < G$, $H' < G'$ are corresponding invariant subgroups, then $H$ has bounded index in $G$ if and only if $H'$ has bounded index in $G'$.
\end{remark}
\begin{proof} We show that $\widetilde{H}(G')$ is a group. $\widetilde{H}(G')$ is closed under taking inverse, because if $\tp(g/A) = \tp(h/A)$, then $\tp(g^{-1}/A) = \tp(h^{-1}/A)$, for $g\in G$, $h\in G'$. We prove that $\widetilde{H}(G')$ is closed under multiplication. Let $a,b \in H(G')$. Consider the type $p(x,y) = \tp(a,b/A)$. Let $(c,d)\in G^2$ realize $p$. Then $c,d$ are in $H$ (because e.g. $\tp(c/A) = \tp(a/A)\in \widetilde{H}$). Therefore $\tp(a,b/A) = \tp(c,d/A)$, so $\tp(a\cdot b/A) = \tp(c \cdot d/A)$. However $c \cdot d \in H$, hence $a\cdot b$ is in $\widetilde{H}(G')$.

For the second part, note that $H$ has bounded index if and only if $G^{\infty}_A \subseteq H$.
\end{proof}

If $G$ is a topological group, then the topological closure $\cl(H)$ of an arbitrary subgroup $H < G$ is also a subgroup. In our setting on the group $G$ (as well as on an arbitrary sort $G^k$) we can consider a compact topology taken from the set of types $S_k(A)$. Namely a subset $C$ of $G^k$ is closed if $C$ is $\bigwedge$-definable over $A$. Under this topology the multiplication $\cdot \colon G \times  G \rightarrow G$ is still continuous, however the topology on $G \times  G$ is not necessarily the product topology (like in Zariski topology on an algebraically closed field). Therefore it is no longer true, that the closure of an $A$-invariant subgroup is a subgroup. The simplest example is the following. Take a big (sufficiently saturated) algebraically closed field $K$ of characteristic $0$. Let $G = (K^{\times},\cdot)$ be the multiplicative group of $K$ with the structure consisting of the traces of all $\emptyset$-definable subset of the field $K$. As a $\emptyset$-invariant subgroup $H$ of $G$ take $H = (2^{\Z},\cdot)$. Then $H$ is not $\bigwedge$-definable, because \[\cl_{\emptyset}(H) = (K\setminus \widehat{\Q}) \cup H,\] where $\cl_{\emptyset}$ is the closure with respect to the topology taken from $\emptyset$-types and $\widehat{\Q}$ is the algebraic closure of $\Q$. In particular $\cl_{\emptyset}(H)$ is not a group (if $t\in K\setminus \widehat{\Q}$, then $t$ and $\frac{3}{t}$ is in $\cl_{\emptyset}(H)$, but $3\not\in \cl_{\emptyset}(H)$).

If $H$ is $A$-invariant subgroup of $G$ with bounded index, then $G^{\infty}_A \subseteq H$. Thus we may work with the image of $H$ in $G/G^{\infty}_A$. Using this approach (and some result on strong types from \cite{gis}) in Theorem \ref{thm:cl} $(ii)$ below, we prove that if additionally $H$ is a normal subgroup, then the topological closure $\cl_A(H)$ of $H$ in $G$ generates the group in 2 steps, i.e. $\cl_A(H)\cdot \cl_A(H)$ is a group. 

$G/G^{\infty}_A$ with the logic topology \cite[Proposition 3.5]{gis} \cite[Definition 3.1]{lascar} is a compact topological group (not necessarily Hausdorff).

\begin{remark}
Let $G'\prec G$ be a small model containing $A$ and \[i \colon G \longrightarrow G/G^{\infty}_A\] be the quotient map. Then $G/G^{\infty}_A$ with the logic topology (a subset $X\subseteq G/G^{\infty}_A$ is closed if and only of its preimage $i^{-1}[X] \subseteq G$ is $\bigwedge$-definable over $G'$) is a compact topological group. Moreover if $G^{00}_A = \bigcap_{j\in J}\varphi_j(G)$, then as a basis of open neighborhoods of the identity in this topology we may take the sets of the form \[ B_j = \{g/G^{\infty}_A : g\cdot G^{00}_A \subseteq \varphi_j(G)\} \text{ for } j\in J.\]
\end{remark}
\begin{proof}
For $j\in J$, $B_j$ is open in $G/G^{\infty}_A$, because \[i^{-1}[B_j^c] = \{g\in G : g\cdot G^{00}_A \cap \neg\varphi_j(G) \neq \emptyset\}.\] If $U\subseteq G/G^{\infty}_A$ is an open neighborhood of the identity $e$, then $\cl(e)\cap U^c = \emptyset$. Therefore $i^{-1}[\cl(e)]\cap i^{-1}[U^c] = \emptyset$ in $G/G^{\infty}_A$. The fact that $i^{-1}[U^c]$ is $\bigwedge$-definable over $G'$ implies existence of $j\in J$ satisfying \[\varphi_j[G]\cap i^{-1}[U^c] = \emptyset.\] This gives $B_j\subseteq U$.
\end{proof}

Applying the previous remark to our context, we see that if $G^{\infty}_A < H < G$ and $H$ is $A$-invariant, then the closure $\cl(i[H])$ in $G/G^{\infty}_A$ is a group. Hence \[i^{-1}\left[\cl(i[H])\right]\] is a $\bigwedge$-definable (over any small model) subgroup of $G$ containing $H$.

\begin{theorem} \label{thm:cl}
Let $H$ be an $A$-invariant subgroup of $G$ of bounded index (hence $G^{\infty}_A \subseteq H$). 
\begin{enumerate}
\item[(i)] $i^{-1}\left[\cl(i[H])\right]$ is the smallest $\bigwedge$-definable subgroup of $G$ containing $H$. In fact, $i^{-1}\left[\cl(i[H])\right]$ is $\bigwedge$-definable over any small model containing $A$.
\item[(ii)] If additionally $H\lhd G$ is a normal subgroup, then $\cl_A(H)$ generates $i^{-1}\left[\cl(i[H])\right]$ in $2$ steps, namely \[i^{-1}\left[\cl(i[H])\right] = \cl_A(H)\cdot X_{\Theta/A}\] is a normal subgroup of $G$.
\end{enumerate}
In particular $G^{00}_{A} = \cl_A(G^{\infty}_{A})\cdot X_{\Theta/A}$.
\end{theorem}
\begin{proof} $(i)$ If $\widetilde{H}$ is a $\bigwedge$-definable subgroup of $G$ containing $H$, then $\widetilde{H}/G^{\infty}_A$ is closed and contains $\cl(i[H])$. Hence $\widetilde{H}\supseteq i^{-1}\left[\cl(i[H])\right]$.

$(ii)$ The group $i^{-1}\left[\cl(i[H])\right]$ is $\bigwedge$-definable over $A$. Therefore $i^{-1}\left[\cl(i[H])\right]$ contains $\cl_A(H)$. Hence, by $(i)$ it is enough to show that $X_{\Theta/A}\cdot\cl_A(H)$ is a group. Essentially this follows from \cite[Theorem 2.3]{gis}. Here are the details. Extend the structure on $G$, to $N = (G,X,\cdot)$, where $\cdot\colon G\times X \rightarrow X$ is a regular action of $G$ on $X$ (considered in \cite[Section 3]{gis}). Without loss of generality assume that $A = \emptyset$ and $N$ is a monster model. By \cite[Proposition 3.3(1)]{gis}, $\aut(N)$ is a semidirect product of $G$ and $\aut(G)$. Namely, first fix an arbitrary point $x_0$ from $X$. Define the following embeddings \[\overline{\cdot}\colon \aut(G) \hookrightarrow \aut(N),\ \  \overline{\cdot}\colon G \hookrightarrow \aut(N),\] as follows: for $h,g\in G, f\in \aut(G)$ let $\overline{f}, \overline{g} \in \aut(N)$ be defined as \[\overline{f}|_{G}=f,\ \overline{f}(h\cdot x_0)=f(h)\cdot x_0,\ \overline{g}|_{G}=\id_{G},\ \overline{g}(h\cdot x_0)=(hg^{-1})\cdot x_0.\] One can easily prove that \[\overline{f}\circ \overline{g} = \overline{f(g)}\circ \overline{f},\ \ \overline{g}\circ \overline{f} = \overline{f}\circ \overline{f^{-1}(g)}, \tag{$\heartsuit$}\] so (see \cite[Proposition 3.3(1)]{gis}) \[\aut(N) = G \rtimes \aut(G)\] (in $\aut(N)$: $G\cap \aut(G) = \{0\}$ and $\aut(G)$ acts on $G$ by conjugation $\overline{g}^{\overline{f}} = \overline{f(g)} \in G$).

Define \[H' = H \rtimes \autfL(G).\] Then $H'$ is a normal subgroup of $\aut(N)$, because by $(\heartsuit)$ for $h\in H$, $e\in\autfL(G)$, $g\in G$, $f\in\aut(G)$: \[(\overline{h} \circ \overline{e})^{\overline{g} \circ \overline{f}} = \overline{f}^{-1} \circ \overline{g^{-1}\cdot h\cdot e(g)}\circ \overline{e\circ f} = \overline{f^{-1}(g^{-1}\cdot h\cdot e(g))}\circ \overline{f^{-1}\circ e\circ f}.\] Now $f^{-1}(g^{-1}\cdot h\cdot e(g)) = f^{-1}(h^g)\cdot f^{-1}(g^{-1} e(g)) \in H^G\cdot X_{E_{\La}} = H$, and $e^f\in\autfL(G)$.

Note that $H'$ contains all Lascar strong automorphisms of $N$, because by \cite[Proposition 3.3(3)]{gis} $\autfL(N) = G^{\infty}_{\emptyset} \rtimes \autfL(G)$.

Using $H'$ we can define an orbit equivalence relation $E_{H'}$ on $N$: \[E_{H'}(a,b)\ \Leftrightarrow \ a=f(b), \text{ for some }f\in H'.\] Applying \cite[Theorem 2.3(ii)]{gis}, to $H'$ we obtain that on the sort $X$ \[E_{\overline{\overline{H'}}} = \Theta \circ \overline{E_{H'}}, \tag{1}\] where $\overline{\overline{H'}}=j^{-1}[\cl(j[H'])]$ and $j\colon \aut(N) \to \galL(N)$ is a quotient map. $\overline{E_{H'}} = \cl_{\emptyset}(E_{H'})$ is the closure of $E_{H'}$ in $S_2^N(\emptyset)$ (the set of 2-types in $N$ over $\emptyset$). 

By $(1)$ and \cite[Lemma 3.7(1)]{gis} we have that \[\overline{\overline{H'}}\cdot x_0 = x_0/E_{\overline{\overline{H'}}} = X_{\Theta}\cdot\{y\in X : \overline{E_{H'}}(y,x_0)\}.\] It is enough to prove that \[\{y\in X : \overline{E_{H'}}(y,x_0)\} = \cl_{\emptyset}(H)\cdot x_0, \tag{2}\] because then, by regularity of the $G$ action on $X$ we conclude that $\overline{\overline{H'}}|_X = X_{\Theta}\cdot\ \cl_{\emptyset}(H)$ is a group.

There is a canonical homeomorphism between $S_2^{X}(\emptyset)$ and $S_1^{G}(\emptyset)$ given by: \[\tp(g_1x_0,g_2x_0) \ \ \longleftrightarrow\ \  \tp(g_1g_2^{-1})\ \  \text{ and } \ \ \tp(x_0,g\cdot x_0)\ \ \longleftrightarrow\ \ \tp(g).\] Namely, let $(x_1,x_2) = (g_1x_0, g_2x_0)\in {X}^2$ and $F = \overline{g}\overline{f}\in\aut(N)$ be an arbitrary automorphism. Then \[F(x_1,x_2) = F(g_1\cdot x_0, g_2\cdot x_0) = (f(g_1)g^{-1}\cdot x_0,f(g_2)g^{-1}\cdot x_0) = (g_1'\cdot x_0, g_2'\cdot x_0),\] where ${g_1'}{g_2'}^{-1} = f(g_1g_2^{-1})$, so ${g_1'}{g_2'}^{-1} \equiv g_1g_2^{-1}$ ($f\in\aut(G)$). On the other hand if $\tp(g_1) = \tp(g_2)$, then for every $x_1,x_2\in X$, \[\tp(x_1, g_1x_1) = \tp(x_2, g_2x_2),\] because $F(x_1, g_1x_1) = (x_2, g_2x_2)$, where $F = \overline{g}\overline{f}$ such that $f(g_1) = g_2$ and $g = {g_2'}^{-1}f(g_1')$.

Under this homeomorphism $E_{H'} \subseteq S_2^{X}(\emptyset)$ corresponds to $H \subseteq S_1^{G}(\emptyset)$. Thus if $y=g\cdot x_0$, then $\overline{E_{H'}}(x_0,y) \Leftrightarrow \tp(x_0,g\cdot x_0) \in \cl_{\emptyset}(E_{H'}) \Leftrightarrow \tp(g) \in \cl_{\emptyset}(H)$ and we get $(2)$.
\end{proof}

\section{Thick sets}

In this section we generally assume (unless otherwise is stated) that $(G,\cdot,\ldots)$ is an arbitrary group (not necessarily a monster model) with some additional $L$-structure. 

We introduce the notion of \emph{thick} set (this conception is based on the definition of thick formula). Using this notion and result from previous section, we give a characterisation (Theorem \ref{thm:opis}) of type definable closure of normal parameter-invariant subgroups of $G$ with bounded index (so also of $G^{00}_A$) and a new description of $X_{\Theta_A}$ (Lemma \ref{lem:thicks}), which will be useful in the next section.

\begin{definition} Let $(G,\cdot,\ldots)$ be a group, $X\subseteq G$ and $n \in \N$. \label{def:thick_set}
\begin{enumerate}
\item[(1)] We say that $X$ is \emph{right [left] $n$-generic} if at most $n$ right [left] translates of $X$ by elements of $G$ cover the whole group $G$. $X$ is \emph{generic} if it is right $n$-generic for some natural $n$.
\item[(2)] We call $X$ \emph{$n$-thick} if $X=X^{-1}$ and \[ (\forall g_0,\dots, g_{n-1}\in G)(\exists  i < j < n)\ \  g_i^{-1}\cdot g_j\in X\] (we do not require $g_0,\dots, g_{n-1}$ to be pairwise distinct). $X$ is \emph{thick} if it is $n$-thick for some natural $n$. The \emph{thickness} of $X$ is the smallest $n$ such that $X$ is $n$-thick.
\item[(3)] For definable set $X\subseteq G$ we define the formula $\varphi_X (x,y) = x^{-1}y\in X$.
\end{enumerate}
\end{definition}

An easy example of thick subset of the group is a subgroup of finite index. In Lemma \ref{lem:thick} below we collect the basic properties of thick sets. 

\begin{lemma} \label{lem:thick}
Let $(G,\cdot,\ldots)$ be a group, $X \subseteq G$ and $\varphi(x,y)\in L$.
\begin{enumerate}
\item[(1)] If $X$ is definable, then $X=X_{\varphi_X}$ and $\varphi \vdash \varphi_{X_{\varphi}}$.
\item[(2)] If $\varphi(x,y)$ is thick, then the set $X_{\varphi}$ and formulas $\varphi_{X_{\varphi}}$, $\varphi(x^{-1},y^{-1})$ are also thick. If $X$ is $A$-definable and $X = X^{-1}$, then $X$ is thick if and only if the formula $\varphi_X\in L(A)$ is thick.
\item[(3)] If $X$ is $A$-definable and thick, then in a sufficiently saturated extension $G^*$ of $G$ for every sequence $(a_i)_{i<\omega} \subseteq G^*$ that is $2$-indiscernible over $A$, $a_i^{-1}a_j\in X$, for every $i<j<\omega$.
\item[(4)] If $X$ is $n$-thick, then $X$ is right and left $n$-generic. If $X$ is right [left] $n$-generic, then $X^{-1}X$ [$XX^{-1}$ respectively] is $(n+1)$-thick.
\end{enumerate}
\end{lemma} 
\begin{proof} $(1)$ and $(2)$ are easy. 

$(3)$ By definition $a_i^{-1}a_j\in X$ holds, for some $i<j<\omega$. By $2$-indiscernibility $a_i^{-1}a_j\in X$ holds for every $i<j<\omega$.

$(4)$ Let $g_0,\ldots,g_{k}\in G$ be a maximal sequence for which $g_i^{-1}g_j \notin X$, for every $i<j<n$. Take an arbitrary $g\in G$ and consider the sequence $g_0,\ldots,g_{k},g$. By assumption there is $i<n$ satisfying $g_i^{-1}g\in X$, thus $g\in g_iX$. Thus $G = \bigcup_{i<k+1} g_iX$ and $G=G^{-1} = \bigcup_{i<k+1} Xg_i^{-1}$ ($X=X^{-1}$).

Assume that $G = \bigcup_{i<n} h_iX$ and $g_0,\ldots,g_{n}\in G$. Then for some $i<j<n$ there is $k<n$ such that $g_i,g_j \in h_kX$. Therefore $g_i^{-1}g_j\in X^{-1}X$.
\end{proof}

The notion of ``genericity'' is classical in model theory. By Lemma \ref{lem:thick}$(4)$, '`thickness'' and ``genericity'' are very close to each other.

Thick formulas over $A$ were needed to define the relation $\Theta_{A}$, the transitive closure of which is $E_{\La/A}$. Similarly, an intersection of all $A$-definable thick sets in a monster model is $X_{\Theta_{A}}$, which generates $G_{\La/A} = G^{\infty}_{A}$ (Lemma \ref{lem:def}$(2)$).

\begin{lemma} \label{lem:thicks}
Let $(G,\cdot,\ldots)$ be a monster model, $A\subset G$ be small and $n \in \N$. Then
 \[X_{\Theta_A}^n = \bigcap \{P^n : P\subseteq G \text{ is $A$-definable and thick}\}.\]
\end{lemma}
\begin{proof}
Let first $n=1$ and $A = \emptyset$. By definition $X_{\Theta} = \{a^{-1}b : \Theta(a,b)\}$. Now it is enough to use Lemma \ref{lem:thick}$(1), (2)$ (every $A$-definable thick set corresponds to a thick formula over $A$) and compactness. The case $n>1$ follows from compactness.
\end{proof}

Using thick sets we can describe (as in Theorem \ref{thm:cl}) the closure of a normal invariant subgroup of bounded index with respect to the topology on $G/G^{\infty}_A$ .

\begin{theorem} \label{thm:opis} Let $(G,\cdot,\ldots)$ be a monster model and $H\lhd G$ be an $A$-invariant normal subgroup of bounded index. Then (using notation from Theorem \ref{thm:cl}) the closure of $H$ with respect to the topology on $G/G^{\infty}_A$ is \[i^{-1}\left[\cl(i[H])\right] = \bigcap\{P\cdot Q : P,Q\subseteq G \text{ are } A\text{-definable thick and } P\supseteq H\cup Q\}.\] In particular $G^{00}_A = \bigcap\{P\cdot Q : P,Q\subseteq G \text{ are } A\text{-definable thick and } P\supseteq G^{\infty}_A \cup Q\}$.
\end{theorem}
\begin{proof} When the set $P$ is $A$-definable and contains $H$, then $P\cup P^{-1}$ is thick (otherwise by compactness, $H$ would be of unbounded index). Therefore by Theorem \ref{thm:cl} and compactness we have 
\begin{align*}
i^{-1}\left[\cl(i[H])\right] &= \cl_A(H)\cdot X_{\Theta_A} \\
&= \bigcap\{P : H\subseteq P \text{ is $A$-definable, thick}\}\cdot \bigcap\{Q : Q \text{ is $A$-definable, thick}\} \\
&= \bigcap\{P\cdot Q : P,Q \text{ are $A$-definable, thick and } P\supseteq H\} \\
&= \bigcap\{P\cdot (Q\cap P) : P,Q \text{ are $A$-definable, thick and } P\supseteq H\}.
\end{align*}
\end{proof}

When $L_0$ is a sublanguage of $L$, then by $G|_{L_0}$ we denote the reduct of $G$ to $L_0$. The next remark is standard. The proof of $(3)$ is new. 

\begin{remark} \label{rem:ccc}
\begin{enumerate}
\item[(1)] $G^{0}_A = \bigcap\left\{(G|_{L_0})^{0}_{A_0} : A_0\subseteq A, L_0\subseteq L \text{ are finite}\right\}$
\item[(2)] $G^{00}_A = \bigcap\left\{(G|_{L_0})^{00}_{A_0} : A_0\subseteq A, L_0\subseteq L \text{ are countable}\right\}$
\item[(3)] $G^{\infty}_A = \bigcap\left\{(G|_{L_0})^{\infty}_{A_0} : A_0\subseteq A, L_0\subseteq L\text{ are countable}\right\}$
\end{enumerate}
\end{remark}
\begin{proof} In all cases inclusion $\subseteq$ is obvious. We prove $\supseteq$.

$(1)$ If $a\not\in G^{0}_A$, then for some $A$-definable $H < G$ of finite index, $a\not\in H$. Take $L_0(A_0)$ as a language in which $H$ is defined. Then $a\not\in (G|_{L_0})^{0}_{A_0}$.

$(2)$ Every $\bigwedge$-definable bounded equivalence relation is a conjunction of a family of thick formulas. Similarly every $\bigwedge$-definable subgroup of bounded index is an intersection of a family of thick sets. Assume that $a\not\in G^{00}_A$ and let $\{Q_j\}_{j<\kappa}$ be the family of $A$-definable thick sets such that \[G^{00}_A = \bigcap_{j<\kappa}Q_j \text{ and } \forall j<\kappa\  \exists i<\kappa\ Q_i^2 \subseteq Q_j.\] Then $a\not\in Q_{j_0}$, for some $j_0<\kappa$. Define by induction a subfamily $\{Q_{j_n}\}_{j<\omega}$ with $Q_{j_{n+1}}^2 \subseteq Q_{j_n}$. Let $L_0(A_0)$ be a countable language over which all $Q_{j_n}$ are definable. Then $a\not\in \bigcap_{n\in \N} Q_{j_n} \supseteq (G|_{L_0})^{00}_{A_0}$.

$(3)$ By description of $G^{\infty}_A$ from Lemmas \ref{lem:def}$(2)$ and \ref{lem:thicks}, condition $a\not\in G^{\infty}_A$ is equivalent to the following
\begin{quote}
for every natural $n$, there is an $A$-definable thick subset $P_n$ of $G$ such that $a\not\in P_n^n$.
\end{quote}
Now if all $P_n$'s are definable in a countable language $L_0(A_0)$, then by the same lemmas $a\not\in (G|_{L_0})^{\infty}_{A_0}$.
\end{proof}

The next lemma shows a relationship between generic sets (so also thick sets by Lemma \ref{lem:thick}$(4)$) and subgroups of finite index.

\begin{lemma} \label{lem:bound}
If $P\subseteq G$ is symmetric and $m$-generic, then $P^{3m-2}$ is a subgroup of $G$ of finite index $\leq m$.
\end{lemma}
\begin{proof} Let $G = \bigcup_{i<m} g_i\cdot P$ and consider the following map $l:G \rightarrow \omega\cup \{\infty\}$: \[ l(g) = \min\{k \in \N : \exists p_0,\ldots,p_{k-1} \in P,\ g=p_0 \cdot \ldots \cdot p_{k-1}\}.\] Note that $l[G]$ is a proper initial interval of natural numbers, possibly enlarged by $\infty$. $l$ has value $1$ on $P$. We may assume that $g_0 = e$. On each set $g_i\cdot P$ the map $l$ has values in $\{k, k+1, k+2,\infty\}$ for some natural $k$, because if $x,y\in g_i\cdot P$, $l(x) = k<\infty$, then $l(g_i)\leq k+1$ and $l(y)\leq k+2$. If $M = \max\{l(g)<\infty : g\in G\}$, then clearly $P^M$ is a group with finite index. Since $P$ is $m$-generic, $M\leq 3(m-1)+1 = 3m-2$.
\end{proof}

\section{An interlude on fields and rings}

In this section we consider model theoretic connected components of additive and multiplicative groups of some special kind of rings (including fields). We also investigate the interplay between the notion of thickness in these groups. Some of our results were inspired by \cite{berg}. 

If $R$ is a ring and $P\subseteq R$, then by $P^{\times} = P \cap R^{\times}$ we denote the set of invertible elements from $P$.

\begin{definition}{\cite[Definition 1.2]{berg}} Let $R$ be an infinite ring. 
\begin{itemize}
\item A ring $R$ is in the class $\U_0$ if there is an infinite subset $S\subset R$, having invertible differences, i.e. for every $s_1 \neq s_2 \in S$, $s_1 - s_2 \in R^{\times}$.
\end{itemize}
The additive group $(R,+)$ is abelian, so it is amenable (see e.g. \cite{wag}), i.e. there is a finitely additive, invariant probability measure $\mu$ on the family of all subsets of $R$.
\begin{itemize}
\item A ring $R$ is in the class $\U$ if $R$ is in the class $\U_0$ and for some finitely additive, invariant probability measure $\mu$ on all subsets of $(R,+)$, $\mu(R^{\times})>0$.
\end{itemize}
\end{definition}

Every infinite field (division ring) is in $\U$. By \cite[Propositions 2.2, 2.5]{berg}, every finite-dimensional algebra over an infinite field (e.g. matrix ring ${\MM}_n(K)$) or more generally, every semilocal ring with no finite nonzero homomorphic images is in $\U$ ($R$ is semilocal if $R/\rad(R)$ is semisimple artinian).

\begin{proposition} \label{prop:add}
Let $R$ be a ring from the class $\U_0$. If $P\subseteq R$ is a thick subset in the sense of the additive group $(R,+)$, then for some $p_0,\ldots,p_{k-1}\in P^{\times}$ \[R = \bigcup_{i<k} \tfrac{1}{p_i} \cdot P.\] In particular $P^{\times}$ is left generic in $R^{\times}$ and $R = \left(P^{\times}\right)^{-1} \cdot P$. Moreover \[R^* = \left(\left({\left(R^*,+\right)^{\infty}_{A}}\right)^{\times}\right)^{-1}\cdot (R^*,+)^{\infty}_{A}, \tag{$\triangle$}\] where $R^*$ is a monster model of an arbitrary first order expansion of $(R,+,\cdot,0,1)$ and $A\subset R^*$ is a small set of parameters.
\end{proposition}

As a consequence we have the following: if $H$ is a proper subgroup of the multiplicative group $(R^{\times},\cdot)$ with finite index, then $H$ is thick in the sense of $R^{\times}$, but not thick in the sense of the additive group $(R,+)$.

\begin{proof} Let $P$ be $n$-thick. Enlarge the ring structure on $R$ by a predicate for $P$ and take a monster model $R^*$. Let $P^*$ correspond to $P$ in $R^*$. By the definition of $\U_0$, there is an infinite indiscernible sequence $(a_i)_{i<\omega}$ in $R^*$ with invertible differences, i.e \[\text{for } i<j<\omega,\ \ a_i - a_j \in {R^*}^{\times}.\] We show that \[R^* = \bigcup_{i<j<n} \frac{1}{a_i - a_j} \cdot P^*.\] Note that by Lemma \ref{lem:thick}$(3)$, $a_i - a_j \in P^*$. Take an arbitrary $x \in R^*$. Since $P^*$ is $n$-thick, applying the definition of thickness to $(a_i\cdot x)_{i<n}$ we have that \[a_ix - a_jx = (a_i - a_j)x \in P^*\] for some $i<j<n$, so $x \in \frac{1}{a_i - a_j}P^*$.

To show the last part, let $\A = \{P\subseteq (R^*,+) : P\text{ is $A$-definable and thick}\}$. Note that by the above and compactness $R^* = \bigcap_{P\in\A}(P^{\times})^{-1}\cdot P = \bigcap_{P\in\A}(P^{\times})^{-1}\cdot\bigcap_{P\in\A}P =$ $((X_{\Theta/A})^{\times})^{-1}\cdot X_{\Theta/A} \subseteq \left(\left({\left(R^*,+\right)^{\infty}_{A}}\right)^{\times}\right)^{-1}\cdot (R^*,+)^{\infty}_{A}$.
\end{proof}

As a corollary to the previous proposition we show that in finite fields thickness of all proper multiplicative subgroups in the sense of the additive group is growing uniformly in the power of the field.

\begin{corollary}
For an arbitrary natural number $N$, for almost all finite fields $\F$ the following holds: every proper subgroup $H < (\F^{\times},\cdot)$ is thick multiplicatively, but not $N$-thick additively.
\end{corollary}
\begin{proof}
Assume that for some $N$ there are infinitely many finite fields $\{\F_n\}_{n<\omega}$ with $N$-thick proper subgroups $H_n<\F^{\times}_n$. Take an ultraproduct $(\F,H) = \prod (\F_n,H_n)/\U$. Then $\F$ is the infinite field and $H<\F^{\times}$ is a proper subgroup $N$-thick in the sense of $(\F,+)$, contradicting the previous proposition.
\end{proof}

In the next proposition we prove analogous to $(\triangle)$ result about multiplicative group of rings from the class $\U$.

\begin{proposition} \label{prop:mult}
Let $R$ be a ring from the class $\U$. If $P\subseteq R^{\times}$ is a right generic subset in the sense of $(R^{\times},\cdot)$, then \[P^{-1}\cdot (P - P) = R.\] Moreover \[R^* = ({R^*}^{\times},\cdot)^{\infty}_{A} - ({R^*}^{\times},\cdot)^{\infty}_{A},\] where $R^*$ is a monster model of an arbitrary first order expansion of $R$ and $A\subset R^*$ is a small set of parameters.
\end{proposition}
\begin{proof}
The idea of the proof is based on arguments from \cite[Theorem 1.3]{berg}. Let $f_0,\ldots,f_{k-1}\in R^{\times}$ satisfy \[R^{\times} = \bigcup_{i<k} Pf_i \tag{1}.\] We claim that there is an infinite sequence $(a_i)_{i<\omega}\subseteq R$ and $p<k$ such that \[\text{for every } i<j<\omega,\ \  a_i-a_j\in Pf_p \tag{2}.\] Let $S\subseteq R$ be (by the definition of $\U$) an infinite subset of $R$ with invertible differences. Clearly, by Ramsey Theorem and $(1)$, there is $p<k$ and $(a_i)_{i<\omega}$, satisfying $(2)$.

Let $\mu$ be a measure from the definition of $\U$. By $(1)$, there is $q<k$ such that $\mu(Pf_q)>0$. For any $x\in R$, there exist $i<j<\omega$ satisfying \[(a_ix+Pf_q)\cap(a_jx+Pf_q)\neq\emptyset,\] because otherwise $\{a_ix+Pf_q\}_{i<\omega}$ forms a family of disjoint set of the same positive measure. Then $(a_i-a_j)x\in (P-P)f_q$, so by $(2)$, $f_px\frac{1}{f_q}\in f_p\cdot \frac{1}{a_i-a_j}(P-P) \subseteq P^{-1}(P-P)$. Hence $R = f_pR\frac{1}{f_q} = P^{-1}(P-P)$.

To show the last part, note that by compactness and the first part we have $R^* =$ $X_{\Theta/A}^2 - X_{\Theta/A}^2$ $= ({R^*}^{\times},\cdot)^{\infty}_{A} - ({R^*}^{\times},\cdot)^{\infty}_{A}$.
\end{proof}

Let $K$ be an infinite field. It seems that $(\triangle)$ from Proposition \ref{prop:add} is the most general result that one can prove about $(K,+)^{\infty}_A$ without any assumptions on the structure of $K$. If we additionally assume that $(K,+)^{\infty}$ exists (e.g. if $K$ has NIP then by Theorem \ref{thm:sh}, $(K,+)^{\infty}$ exists, e.g. reals $(\R,+,\cdot,0,1)$ has NIP), then \[(K,+)^{\infty} = K,\] because $(K,+)^{\infty}$ is a nontrivial ideal of $K$: for every $a\in K$, $a\cdot(K,+)^{\infty}$ is an $a$-invariant subgroup of $(K,+)$, isomorphic to $(K,+)^{\infty}$, thus with bounded index, so $(K,+)^{\infty} = x\cdot(K,+)^{\infty}$ (implication: existence of $(K,+)^{00}$ implies $(K,+)^{00} = K$, was also noted by Pillay in \cite[Proposition 4.1$(i)$]{ekp}). These considerations can be generalized to rings from the class $\U$.

\begin{proposition}
Let $R$ be a ring from the class $\U$. If $(R^*,+)^{\infty}$ exists (where $R^*$ is a monster model of some first order expansion of $R$), then \[(R^*,+)^{\infty} = R^*.\] Similarly for $(R^*,+)^{00}$ and $(R^*,+)^{0}$, e.g. if $(R^*,+)^{00}$ exists, then $(R^*,+)^{00} = R^*$.
\end{proposition}
\begin{proof} Let $G$ be $(R^*,+)^{\infty}$, $(R^*,+)^{00}$ or $(R^*,+)^{0}$ (whenever exist). For every $r\in {R^*}^{\times}$, $f(x)=r\cdot x$ is an automorphism of $(R^*,+)$, so $r\cdot G = G$. By the definition of the class $\U$, in $R^*$ there is an infinite set with invertible differences. Therefore, each thick subset of $(R^*,+)$ contains some invertible element. Thus, by compactness $G \cap {R^*}^{\times} \neq \emptyset$ ($(R^*,+)^{00}$ and $(R^*,+)^{0}$ are intersections of definable thick subsets). Therefore ${R^*}^{\times} \subseteq G$. By Proposition \ref{prop:mult}, $G = R^*$.
\end{proof}

In a forthcoming paper \cite{gis2} we prove, that every group $G$ with the following property (called absolutely connectedness) is perfect: for every first order expansion $(G,\cdot,\ldots)$ of $G$, working in a saturated extension, $G^{\infty}$ exists and $G^{\infty}=G$. Therefore, for each ring $R$ from $\U$, there is a first order structure on it, such that $(R^*,+)^{\infty}_A \neq R^*$, but $\left(\left({\left(R^*,+\right)^{\infty}_{A}}\right)^{\times}\right)^{-1}\cdot (R^*,+)^{\infty}_{A} = R^*$.

\section{Canonical connected components and groups with NIP}

In this section we assume that $(G,\cdot,\ldots)$ is a sufficiently saturated model, i.e. $\overline{\kappa}$-saturated and $\overline{\kappa}$-strongly homogeneous for some large cardinal $\overline{\kappa}$.

We analyse connected components of dependent groups i.e. with theories without the independence property, (NIP). All stable and $o$-minimal theories have NIP. Also the theory of algebraically closed valued fields (ACVF) has NIP. Simple unstable theories have the independence property (see \cite{sh}).

Recall (Definition \ref{def:kp}) that if for every small set of parameters $A\subset G$, $G^{\infty}_A = G^{\infty}_{\emptyset}$, then we say that $G^{\infty}$ exists and define it as $G^{\infty}_{\emptyset}$ (similarly for $G^{00}$ and $G^{0}$). The existence of $G^0, G^{00}$ and $G^{\infty}$ is the property of the theory $\Th(G,\cdot,\ldots)$ and does not depend on particular choice of the monster model $G$. In the case of $G^0$ and $G^{00}$ it is folklore. In the case of $G^{\infty}$ it follows from Remark \ref{rem:ccc}$(3)$ (it is also pointed in \cite{newelski3}).

The next fact is known (see e.g. \cite{peter}), but for completeness of the exposition we include a proof.

\begin{remark} \label{rem:folk}
\begin{enumerate}
\item[(1)] If $G^{\infty}$ exists, then $G^{00}$ exists.
\item[(2)] If $G^{00}$ exists, then $G^{0}$ exists.
\end{enumerate}
\end{remark}
\begin{proof} Let $A$ be a small subset of $G$.

$(1)$  Consider $G^{00}_A$. Since $G^{\infty} \subseteq G^{00}_A$, $G^{00}_A$ is the union of cosets of $G^{\infty}$. But there are boundedly many ($<2^{|L(A)|}$) cosets of $G^{\infty}$ in $G$. Therefore the group $D = \bigcap_{f\in\aut(G)}G^{00}_{f[A]}$ has bounded index. $D$ is also $\emptyset$-invariant and $\bigwedge$-definable over $A$, so $D$ is $\bigwedge$-definable over $\emptyset$. Therefore $D = G^{00}_{\emptyset}$.

$(2)$ Consider $D = \bigcap_{f\in\aut(G)}G^{0}_{f[A]}$. Similarly as in (i), $D$ has bounded index and is $\bigwedge$-definable over $\emptyset$. Therefore \[ D = \bigcap_{i<|L(\emptyset)|}P_i,\] where $P_i \subseteq G$ is $\emptyset$-definable and thick subset of $G$ (otherwise $[G:D]$ is unbounded). If $H < G$ is $f[A]$-definable subgroup of $G$ of finite index (where $f\in\aut(G)$), then there is $i<|L(\emptyset)|$ such that $P_i \subseteq H$ and $P_i$ is $n$-thick (for some natural $n$). By Lemmas \ref{lem:thick}$(4)$ and \ref{lem:bound}, $P_i^{3n+1}$ is $\emptyset$-definable subgroup of $H$ of finite index. Hence $G^{0}_{\emptyset} \subseteq D$.
\end{proof}

\begin{definition}[\cite{sh}]
An $L$-theory $T$ is said to have the NIP (for ``not the independence property'') if there is no formula $\varphi(\overline{x},\overline{y}) \in L$ and elements $\{\overline{a}_i, \overline{b}_{w} : i<\omega, \text{ finite }w\subset \omega\}$ from the monster model $\C\models T$ such that \[\C \models \varphi(\overline{a}_i,\overline{b}_{w})\  \Longleftrightarrow\  i \in w,\] i.e. a family of $\varphi$-definable sets $\{\varphi(\overline{a}_i,\C) : i<\omega\}$ is independent in the sense that every boolean combination of elements from the family is nonempty.
\end{definition}

In \cite{sh1,sh2} Shelah proved that for such groups $G^{00}$ exists (even in more wider context of $\bigwedge$-definable groups) and if additionally $G$ is abelian, then $G^{\infty}$ exists. In Theorem \ref{thm:sh} below we strengthen this result showing existence of $G^{\infty}$ without the commutativity assumption.

\begin{theorem} \label{thm:sh}
If a group $G$ has the theory with NIP, then $G^{\infty}$ exists.
\end{theorem}
\begin{proof}
The proof in \cite{sh2} uses the commutativity assumption only in the proof of Main Claim 1.7. Thus first we prove this claim in the general case (we replace the constant 2 from \cite{sh2} by 4) and then we outline the rest of the proof from \cite{sh2} with small changes.
\begin{claim}
For every natural $m$ and $\alpha < \overline{\kappa}$ there exists $\lambda < \overline{\kappa}$ and a family of subsets $\{A_i : i<\lambda\}$ of $G$, $|A_i|<\alpha$ such that \[\text{for every }A\subset G\text{ with }|A|<\alpha,\ \  \bigcap_{i<\lambda}\left(X_{\underset{A_i}{\equiv}}\right)^m \subseteq \left(X_{\underset{A}{\equiv}}\right)^{m+4}. \tag{$\clubsuit$}\] 
\end{claim}

Recall that $X_{\underset{A}{\equiv}} = \{ a^{-1}b : a,b\in G, \tp(a/A) = \tp(b/A)\}$, $B^G = \bigcup_{g\in G}B^g$ (for $B\subseteq G$) and $G$ is $\overline{\kappa}$-saturated and $\overline{\kappa}$-strongly homogeneous.

\begin{proof}[Proof of the Claim]
Suppose the assertion of the claim is false. Then for some natural $m$ and $\alpha < \overline{\kappa}$, by induction we can find a sequence $(A_i,c_i)_{i<\overline{\kappa}}$, satisfying for every $i<\overline{\kappa}$:
\begin{enumerate}
\item[(1)] $A_i \subset G$, $|A_i|<\alpha$, $c_i\in G$,
\item[(2)] $c_i \in \bigcap_{j<i}X_{\underset{A_j}{\equiv}}^m \setminus X_{\underset{A_i}{\equiv}}^{m+4}$.
\end{enumerate}
We may assume that $\overline{\kappa}$ is large enough to apply Erd\H{o}s-Rado Theorem to get (as in \cite{sh2}) an infinite indiscernible sequence $(A_i,c_i)_{i<\omega}$ satisfying $(1)$ and $(2)$.

Now, for a finite set of formulas $\Phi(x,\overline{y}) = \{\varphi_1(x,\overline{y}), \ldots, \varphi_k(x,\overline{y})\}\subset L(\emptyset)$ and small indexed set $\overrightarrow{A}$, we define an equivalence relation \[E_{\Phi/\overrightarrow{A}}(x,y) = \bigwedge_{1\leq i \leq k} \left(\varphi(x,\overrightarrow{A})\leftrightarrow\varphi(y,\overrightarrow{A})\right).\] It is easy to see that for every natural $r$: \[ \underset{A}{\equiv}\ \   = \bigcap_{\Phi\subset L(\emptyset)} E_{\Phi/\overrightarrow{A}},\ \ \ X_{\underset{A}{\equiv}}^r\ \  = \bigcap_{\Phi\subset L(\emptyset)} X_{E_{\Phi/\overrightarrow{A}}}^r,\ \text{ and }\ X_{\underset{A}{\equiv}}^G = \bigcap_{\Phi\subset L(\emptyset)} X_{E_{\Phi/\overrightarrow{A}}}^G. \tag{\textdagger}\] By Remark \ref{rem:e}$(1)$ we can conclude that every conjugate of $X_{\underset{A}{\equiv}}$ is in $X_{\underset{A}{\equiv}}^2$: \[(X_{\underset{A}{\equiv}})^G \subseteq X_{\underset{A}{\equiv}}^2. \tag{\textdaggerdbl}\]

By indiscernibility of $(A_i,c_i)_{i<\omega}$, $(2)$, (\textdagger), (\textdaggerdbl)\  and compactness we can find two finite sets of formulas $\Phi(x,\overline{y}), \Phi'(x,\overline{y}) \subset L(\emptyset)$, such that for every natural $i$:
\begin{enumerate}
\item[(3)] $c_i \not\in X_{E_{\Phi/\overrightarrow{A}_i}}^{m+4}$,
\item[(4)] $(X_{E_{\Phi'/\overrightarrow{A}_i}})^G \subseteq X_{E_{\Phi/\overrightarrow{A}_i}}^2$.
\end{enumerate}

Define for an arbitrary finite sequence $I = (i_1,\ldots,i_n)$ of pairwise distinct elements of $\omega$, the following elements of $G$:
\begin{itemize}
\item $c_{I,0} = c_{2i_1+1}\cdot\ldots\cdot c_{2i_n+1}$,
\item $c_{I,1} = c_{2i_1}\cdot\ldots\cdot c_{2i_n}$.
\end{itemize}

To obtain a contradiction with NIP it is sufficient to show the following:
\begin{enumerate}
\item[(5)] if $j\not\in I$, then $c_{I,0}c_{I,1}^{-1}\in X_{\underset{A_{2j}}{\equiv}} \subseteq X_{E_{\Phi'/\overrightarrow{A}_{2j}}}$,
\item[(6)] if $j\in I$, then $c_{I,0}c_{I,1}^{-1}\not\in X_{E_{\Phi'/\overrightarrow{A}_{2j}}}$.
\end{enumerate}

If $j\not\in I$, then $c_{I,0}$ and $c_{I,1}$ have the same type over $A_{2j}$, thus $c_{I,0}c_{I,1}^{-1}\in X_{\underset{A_{2j}}{\equiv}}$ and (5) follows.

Assume by way of contradiction that (6) does not hold. Thus for some $j\in I$, we have \[c_{I,0}c_{I,1}^{-1}\in X_{E_{\Phi'/\overrightarrow{A}_{2j}}}.\] Let $I = I_1^\smallfrown \{j\}^\smallfrown I_2$, then 
\begin{equation}
\begin{array}{rl}
c_{I,0}\cdot c_{I,1}^{-1} &= c_{I_1,0}\cdot c_{2j+1}\cdot c_{I_2,0}\cdot c_{I_2,1}^{-1}\cdot c_{2j}^{-1}\cdot c_{I_1,1}^{-1} \\
c_{I,0}\cdot c_{I,1}^{-1}\cdot c_{I_1,1}\cdot c_{2j} &= c_{I_1,0}\cdot c_{2j+1}\cdot c_{I_2,0}\cdot c_{I_2,1}^{-1} \\
c_{2j} &= c_{I_1,1}^{-1}\cdot c_{I,1}\cdot c_{I,0}^{-1}\cdot c_{I_1,0}\cdot c_{2j+1}\cdot c_{I_2,0}\cdot c_{I_2,1}^{-1} \\
&= [c_{I_1,1}^{-1} (c_{I,1}c_{I,0}^{-1}) c_{I_1,1}] \cdot (c_{I_1,1}^{-1}c_{I_1,0})\cdot c_{2j+1}\cdot (c_{I_2,0}c_{I_2,1}^{-1}). \nonumber
\end{array}
\end{equation}

Since $j\not\in I_1\cup I_2$ by (5) we have: $(c_{I_1,1}^{-1}c_{I_1,0}), (c_{I_2,0}c_{I_2,1}^{-1}) \in X_{\underset{A_{2j}}{\equiv}}$. By assumptions $c_{I,0}c_{I,1}^{-1}\in X_{E_{\Phi'/\overrightarrow{A}_{2j}}}$ and $c_{2j+1}\in X_{\underset{A_{2j}}{\equiv}}^m $. Therefore using (4) we obtain \[c_{2j} \in X_{E_{\Phi'/\overrightarrow{A}_{2j}}}^{c_{I_1,1}}\cdot X_{\underset{A_{2j}}{\equiv}}^{m+2} \subseteq  X_{E_{\Phi/\overrightarrow{A}_{2j}}}^{m+4},\] contrary to $(3)$.
\end{proof}

Now we proceed as in \cite{sh2}, with small changes. Let $B\subset G$ be an arbitrary small set of parameters. We will prove that $G^{\infty}_B = G^{\infty}_{\emptyset}$. Let \[B' = B \cup \{\text{representatives of all classes of }E_{\La/B}\}, \tag{$\spadesuit$}\] and $\alpha = |B'|^{+}\leq \left(2^{|L(B)|}\right)^{+}$. By the Claim, for every natural $m$ we have a number $\lambda_m<\overline{\kappa}$ and a family $\{A_i\}_{i<\lambda_m}$ satisfying $(\clubsuit)$. Since the cofinality $\cf(\overline{\kappa}) > \omega$, we may find a universal number $\lambda = \sum_{m<\omega}\lambda_m <\overline{\kappa}$ and a family $\A = \{A_i\}_{i<\lambda}$ satisfying $(\clubsuit)$. We may also assume that $\A$ is closed under finite unions. Define
\[S_m = \bigcap_{i<\lambda} \left(X_{\underset{A_i}{\equiv}}\right)^m,\ \ T_m = \bigcap_{A\subset G, |A|<\alpha} \left(X_{\underset{A}{\equiv}}\right)^m.\]
Then the following hold for every natural $m$
\begin{enumerate}
\item[(1)] $T_m\subseteq S_m\subseteq T_{m+4}$,
\item[(2)] $S_m = (S_1)^m, T_m = (T_1)^m$,
\item[(3)] $T_1 \subseteq X_{\underset{B'}{\equiv}} \subseteq X_{\La/B} = \{a^{-1}b:E_{\La/B}(a,b)\},$ (see Remark \ref{def:la}),
\item[(4)] $X_{\Theta/\bigcup_{i<\lambda}A_i} \subseteq S_1$.
\end{enumerate}
$(1)$ follows from the Claim, $(2)$ follows from compactness and the fact that $\A$ is closed under finite unions, $(3)$ follows from $(\spadesuit)$, $(4)$ is easy.

Let \[H = \bigcup_{m<\omega} S_m = \bigcup_{m<\omega} T_m.\] By $(2)$, $H$ is a group. Since all $T_m$'s are $\emptyset$-invariant, $H$ is $\emptyset$-invariant too. By $(4)$, $H$ contains $G^{\infty}_{\bigcup_{i<\lambda}A_i}$, so has bounded index. Thus $G^{\infty}_{\emptyset} \subseteq H$. By $(3)$, $H\subseteq G^{\infty}_B$. Therefore \[G^{\infty}_{\emptyset} \subseteq H \subseteq G^{\infty}_B \subseteq G^{\infty}_{\emptyset},\] hence $G^{\infty}_B = G^{\infty}_{\emptyset}$.
\end{proof}

The converse of Theorem \ref{thm:sh} is not true. Theorem 5.1 from \cite{jal} says that an existentially closed $CSA_f$-group $(G,\cdot)$ (with the pure group structure) is definably simple (for every nontrivial $a,b\in G$, $a\in \left(b^G\cup {b^{-1}}^G\right)^3$). Thus a sufficiently saturated extension $G^*$ of $G$ is also simple and ${G^*}^{\infty}$ exists and equals to $G^*$. On the other hand by \cite[Corollary 2.2]{jal2} we have that the theory of $G$ has the independence property.

In \cite{NIP2} the authors proved that if the group $G$ has theory with NIP and is definably amenable (there exists a left invariant finitely additive measure on definable sets) then $G^{\infty} = G^{00} = \stab(p_0)$ for some global type $p_0\in S_1(G)$ consisting of formulas with positive measure. Therefore we may ask, in how many steps in this case $X_{\Theta}$ generates $G^{\infty}$? The answer is 2. This particular type $p_0$ satisfies assumptions from the next proposition.

\begin{proposition} \label{prop:ogr}
If there exists a global type $p\in S_1(G)$ such that the orbits $G\cdot p$ and $\aut(G) \cdot p$ are bounded, then $G^{\infty}$ exists and \[G^{\infty} = G^{00} = X_{\Theta}^2 = \stab^G(p).\]
\end{proposition}
\begin{proof}
Let $M \subset G$ be a small model. It is easy to see that (working in bigger monster model $G'\succ G$): \[\stab^G(p) = \{g\in G : g\cdot p = p\} \subseteq X_{\underset{G}{\equiv}} \subseteq X_{\Theta_M}^2 \subseteq G^{\infty}_M.\]
The group $\stab^G(p)$ has bounded index, because $[G:\stab^G(p)] = |G\cdot p|$.
This group is also $\emptyset$-invariant, because $H = \bigcap_{f\in\aut(G)}\stab^G(f(p))$ is in fact an intersection of boundedly many groups ($\aut(G) \cdot p$ is bounded). Thus \[G^{\infty}_{\emptyset} \subseteq H \subseteq\stab^G(p),\] so $G^{\infty}_M = G^{\infty}_{\emptyset} = X_{\Theta_M}^2$ is $\bigwedge$-definable.
\end{proof}

For example if the group $G$ is definably compact, definable in an $o$-minimal expansion of a real closed field, then Proposition \ref{prop:ogr} applies, because by \cite{NIP} $G$ has boundedly many generic types (the property of being generic is preserved under action of $G$ and $\aut(G)$).

\begin{example} \label{ex:s1}
It is not true that under the assumption of Proposition \ref{prop:ogr}, $G^0 = G^{00}$. Consider the circle group $G = S^1=([0,1),+_{(\text{mod } 1)})$ with the structure from the ordered field of reals $\R$. Let $G^*$ be the monster model. $G^*$ is definably compact, so Proposition \ref{prop:ogr} applies and ${G^*}^{\infty} = {G^*}^{00}$ is the subgroup of infinitesimal elements (in a sufficiently saturated extension). However $G^*$ has no subgroups of finite index because is divisible, thus $G^* = {G^*}^{0} \neq {G^*}^{00} = {G^*}^{\infty}$.
\end{example}

\noindent Instytut Matematyczny Uniwersytetu Wroc{\l}awskiego\\
pl. Grunwaldzki 2/4, 50-384 Wroc{\l}aw, Poland\\
and\\
Mathematical Institute of the Polish Academy of Sciences\\
\bigskip
ul. \'Sniadeckich 8, 00-956 Warsaw, Poland\\
e-mail: gismat@math.uni.wroc.pl,\\
\bigskip
www: www.math.uni.wroc.pl/\~{}gismat\\
mailing address:\\
Instytut Matematyczny Uniwersytetu Wroc{\l}awskiego\\
pl. Grunwaldzki 2/4, 50-384 Wroc{\l}aw, Poland
\end{document}